\documentclass{amsart}
\usepackage[latin1]{inputenc}
\usepackage[T1]{fontenc}
\usepackage{lmodern}
\usepackage[english]{babel}
\usepackage{microtype}

\usepackage{amsmath,amssymb,amsfonts,amsthm}
\usepackage{mathtools,accents}
\usepackage{mathrsfs}
\usepackage{aliascnt}
\usepackage{braket}
\usepackage{bm}
\usepackage{lineno}
\usepackage[citecolor=blue,colorlinks]{hyperref}

\usepackage{enumerate}
\usepackage{xcolor}

\usepackage{aliascnt}

\makeatletter
\def\newaliasedtheorem#1[#2]#3{
  \newaliascnt{#1@alt}{#2}
  \newtheorem{#1}[#1@alt]{#3}
  \expandafter\newcommand\csname #1@altname\endcsname{#3}
}
\makeatother

\theoremstyle{plain}
\newtheorem{theorem}{Theorem}[section]
\newaliasedtheorem{lemma}[theorem]{Lemma}
\newaliasedtheorem{prop}[theorem]{Proposition}
\newaliasedtheorem{claim}[theorem]{Claim}
\newaliasedtheorem{corollary}[theorem]{Corollary}

\theoremstyle{definition}
\newaliasedtheorem{definition}[theorem]{Definition}
\newaliasedtheorem{example}[theorem]{Example}

\theoremstyle{remark}
\newaliasedtheorem{remark}[theorem]{Remark}

\numberwithin{equation}{section}

\def\Im{\textrm{Im}}
\def\Re{\textrm{Re}} 
\def\11{{\rm 1~\hspace{-1.4ex}l} }
\def\R{\mathbb R}

\def\N{\mathbb N}

\title
[On the asymptotic behavior of moments for NLS]
{On the asymptotic behavior of high order moments for a family of Schr\"odinger equations}

\author[N. Tzvetkov \and N. Visciglia]
{N. Tzvetkov \and N. Visciglia}

\address{N.~Tzvetkov,
Universit\'e de Cergy-Pontoise,  Cergy-Pontoise, F-95000,UMR 8088 du CNRS}
\email{nikolay.tzvetkov@u-cergy.fr}

\address{N.~Visciglia, Dipartimento di Matematica, Universit\`a di Pisa, Largo Bruno Pontecorvo, 5, 56100 Pisa, Italy}
\email{nicola.visciglia@unipi.it}

\begin{document}

\maketitle

\begin{abstract}
We study upper bounds and the asymptotic behavior of high order moments for solutions to a family of linear and nonlinear Schr\"odinger equations. 
\end{abstract}

\section{Introduction}
Consider the free Schr\"odinger equation 
\begin{equation}\label{free}
\begin{cases}
(i\partial_t+\partial_x^2)u=0,\quad (t,x)\in \R\times \R \\
u(0,x)=f\in \Sigma_s(\R), \,\, s\in \N, \,\, s\geq 1,
\end{cases}
\end{equation}
where 
\begin{equation}\label{norms}
\|f\|_{\Sigma_s(\R)}^2=\|f\|_{H^s(\R)}^2+ \int_{\R} x^{2s} |f(x)|^2dx \,.
\end{equation}  
Here $H^s(\R)$ denote the classical Sobolev spaces.  The solution of \eqref{free} can be written as 
\begin{equation}\label{formula}
u(t,x)=\frac{1}{\sqrt{2\pi}}\int_{\R}e^{ix\xi-it \xi^2}\hat{f}(\xi)d\xi,
\end{equation}
where 
\begin{equation}\label{fourier}\hat{f}(\xi)=\frac 1 {\sqrt {2\pi}} \int_\R e^{-i x\xi } f(x) dx
\end{equation}
denotes the Fourier transform of $f\in \Sigma_s(\R),$  $s\geq 1$.  It follows directly from \eqref{formula} that 
\begin{equation}\label{PR1}
\|u(t,\cdot)\|_{H^s(\R)}=\|f\|_{H^s(\R)},\quad \forall\, t\in\R\,.
\end{equation}

Another direct consequence of \eqref{formula} is the invariance of the space $\Sigma_s(\R)$
by the linear flow associated with \eqref{free}, namely $e^{it\partial_x^2 } f\in \Sigma_s(\R)$ for every $t$ with the following quantitative bound:
\begin{equation}\label{basc}\int_\R x^{2s} |u(t,x)|^2 dx\lesssim \langle t\rangle^{2s}.
\end{equation}
 Moreover one can show the exact long-time behavior
\begin{equation}\label{PR2}
\lim_{t\rightarrow \pm\infty}\int_{\R}
\big(\frac xt\big)^{2s} |u(t,x)|^2 dx=2^{2s}\int_{\R}  \big|\partial_x^s f\big|^2 dx \,.
\end{equation}
Property \eqref{PR1} implies that for the solutions of \eqref{free} no migration of Fourier modes is possible.
On the other hand property \eqref{basc} and \eqref{PR2} imply that there is a migration of the conserved $L^2$ mass of the solutions of \eqref{free}  to the spatial infinity.
\\

Our goal in this work is to show that \eqref{basc} and more importantly \eqref{PR2} persist for a large class of perturbations of \eqref{free}, both linear and nonlinear. 
We therefore consider the following family of defocusing NLS:
\begin{equation}\label{NLS}
\begin{cases}
\big(i\partial_t  + \partial_{x}^2  + V(x)+\lambda |u|^{2k}\big) u=0, \quad
(t,x)\in \R\times \R, \lambda\leq 0, k\in \N, k\geq 2,\\
u(0, x)=f\in \Sigma_s(\R), \quad s\in \N, \quad s\geq 1.
\end{cases}
\end{equation}
We suppose that $V:\R\rightarrow \R$ is a potential such that:
 \begin{equation}\label{boundedreivatives}
 |\partial_x^j V(x)|\lesssim (1+ |x|)^{-1},  \quad  \forall j=1,\cdots , s.
 \end{equation}
 Notice that we have imposed decay on the derivatives of $V$ and not on the potential $V$ itself. Indeed in the sequel we shall assume either $V\in L^\infty(\R)$
or $V\in L^\infty(\R)$ and $\lim_{x\rightarrow \infty} V=0$.
It will depend on the kind of results we are looking for, and in any case it will be specified 
along the statements.
Since we shall work with the perturbed Sobolev spaces, we need to compute
the operator $\sqrt{-\partial_x^2 - V}$, hence we shall assume
\begin{equation}\label{posit}
-\partial_x^2 - V\geq 0 \hbox{ in the operator sense. }
\end{equation}
As mentioned above, our aim is the analysis of the invariance of the 
space $\Sigma_s(\R)$ along the flow associated with \eqref{NLS} with a quantitative estimate, as well as the asymptotic description of the moments
$$\int_\R x^{2s} |u(t,x)|^2 dx \hbox{ as } t\rightarrow \pm \infty.$$
Notice that for $\lambda=0$ the equation \eqref{NLS} reduces to the purely linear Schr\"odinger equation with a linear perturbation $V$, while for $\lambda<0$ it is nonlinear, defocusing and not translation invariant, unless $V\equiv 0$.
It is worth mentioning that, at the best of our knowledge, very few results are known about the growth of higher order
moments, in the linear case either. We quote in this direction \cite{DGS},
\cite{ES},  \cite{RS}, \cite{S}, \cite{Z} and the references therein. In the aforementioned papers the results are mainly devoted to the study of qualitative informations on the growth of high order moments, but we are not aware of any result where it is studied its
asymptotic behavior for large times.
\\

Our first result provides a general statement about upper bounds (and not yet on the asymptotic behavior) 
of the moments of order $s$ for a large class of linear potential perturbations.
At the best of our knowledge the result below is not stated elsewhere in the literature
under its full generality.
\begin{theorem}\label{mainlin}
Let $s\geq 1$ be a fixed integer.
Assume that $V$ is real valued, $V\in L^\infty(\R)$ and satisfies \eqref{boundedreivatives}, \eqref{posit}.
Then 
for every $f\in \Sigma_s(\R)$, 
we get the following upper-bound:
\begin{equation}\label{momentssV}
\|e^{it (\partial_x^2 + V)} f\|_{\Sigma_s(\R)}^{2}\lesssim 
\langle t \rangle^{2s}.\end{equation}
\end{theorem}

Concerning the translation invariant nonlinear case   
(namely \eqref{NLS} with $V=0$ and $\lambda\neq 0$) we quote the paper \cite{carles} where it is shown that the moments of order $s$ grow as $t^s$, however the limit as time goes to $\pm \infty$ is not explicitely studied (even if it should follow by a simple extra argument, see Remark~\ref{carlino}). However, the technique used in \cite{carles} seems to be specific for the case $V=0$. More precisely, it is based on the commuting vector fields approach which is not available in the case $V\neq 0$. At this point 
we should mention \cite{CGV} where a suitable pseudo-differential version of vector fields is constructed under a potential perturbation and as a result it is obtained the time decay of the corresponding nonlinear solutions. 
However it is unclear how to use the technique developed in \cite{CGV} to study the behavior of high order moments in the nonlinear setting.
\\

In order to state the nonlinear counterpart of Theorem \ref{mainlin}
(namely for solutions to \eqref{NLS} with $\lambda< 0$) 
first of all we recall  that, due 
to the defocusing character of \eqref{NLS} and due to the classical Sobolev embedding $H^1(\R)\subset L^\infty(\R)$, there exists
an unique global solution $u(t,x)\in {\mathcal C}(\R; H^1(\R))$ 
for every $f\in H^1(\R)$. Therefore from now on we shall not comment anymore about the existence and uniqueness of solution to
\eqref{NLS} for initial data $f\in \Sigma_s(\R)$,  $s\geq 1$.
However even if  $f\in \Sigma_s(\R)$, it is not guaranteed for free that the solution $u(t,x)$ belongs to $\Sigma_s(\R)$ for every time $t\in \R$. This is part of our work, 
along with quantitative information on the behavior of the high order moments for large time.
\\

Next, we introduce a further assumption that we shall use in the nonlinear setting:
\begin{equation}\label{decay} 
\sup_{(t, x)\in \R^2}\,  \langle t\rangle^{\frac{1}{2}} |u(t,x)|<\infty\,,
\end{equation}
where $u(t,x)\in {\mathcal C}(\R; H^1(\R))$ is the unique solution
to \eqref{NLS} with $f\in H^1(\R)$.
\\

We point out  that assumption \eqref{decay} is satisfied for instance in the case
$\lambda<0$ and for any repulsive potentials $V(x)$, namely:
\begin{equation}\label{repulsive}2V(x) + x\partial_x V(x)\geq 0.\end{equation}
In fact under this assumption, by the pseudoconformal energy estimate (see 
\cite{GV} and \cite{Ca}), one can show:
\begin{equation}\label{Jprime}\|xu(t,x)+ 2it \partial_x u(t,x)\|_{L^2({\R})}\lesssim 1.\end{equation}
On the other hand by the Gagliardo-Nirenberg inequality we get
\begin{multline*}
\sqrt {2|t|} \|e^{-i\frac{x^2}{4t}}u(t,x)\|_{L^\infty({\R})}
\\\lesssim \sqrt {2|t|} \|e^{-i\frac{x^2}{4t}}u(t,x)\|_{L^2(\R)}^{\frac {1}{2}}
\|\partial_x (e^{-i\frac{x^2}{4t}}u(t,x))\|_{L^2(\R)}^{\frac{1} {2}}
\\= 
\|u(t,x)\|_{L^2({\R})}^{\frac 1 2}
{\|xu(t,x)+ 2it \partial_x u(t,x)}\|_{L^2({\R})}^{\frac 1 2} \lesssim 1,
\end{multline*}
where we used \eqref{Jprime} and the $L^2$ conservation at the last step.
Hence \eqref{decay} follows under the assumption \eqref{repulsive}.
\\

We also recall that \eqref{decay} has been established in the paper \cite{CGV} for a large class of potentials under a smallness assumption on the initial datum
and with a generic spectral condition on the potential $V$.
\\

Next we give our statement about the invariance of the space $\Sigma_s(\R)$ along with 
an upper bound in
the defocusing nonlinear case (namely \eqref{NLS} with $\lambda<0$).
\begin{theorem}\label{mainnonlin}
Let $s\geq 1$, $k\geq 2$ be two fixed integers, and let $\lambda< 0$.
Assume that the potential $V$ is real valued, $V\in L^\infty(\R)$ and satisfies \eqref{boundedreivatives}, \eqref{posit}. Let $u(t,x)$ satisfies
\eqref{decay}, where $u(t,x)\in \mathcal(\R;H^1(\R))$ is the unique global solution to \eqref{NLS} ($\lambda, k$ as above), with $f\in \Sigma_s(\R)$.
Then we get the following upper bound:
\begin{equation}\label{momentss}\|u(t,x)\|_{\Sigma_s(\R)}^{2}\lesssim 
\langle t\rangle^{2s}.\end{equation}
\end{theorem}
Notice that in Theorem~\ref{mainnonlin}, compared with Theorem~\ref{mainlin}, we assume 
the extra condition \eqref{decay} which is not needed in the linear case.
Nevertheless the conclusion about the moments that we get are the same.
\\
Concerning Theorems \ref{mainlin} and \ref{mainnonlin} we shall provide one unified proof. Indeed it will be clear along the proof that
one can drop the assumption \eqref{decay}
in the linear case $\lambda=0$.
In fact this assumption is needed 
to deal with the nonlinear term, but it is not needed to deal with the linear part.\\
\\

Next we focus on the issue of the precise description
of the moments
of order $s$ for large times.
In order to do that we need another assumption on the potential $V$.
More specifically we assume the asymptotic completeness of the wave operator
in the $L^2$ setting:
\begin{equation}\label{reedsimon}
\forall g\in L^2(\R) \quad \exists\, g_\pm\in L^2(\R) \hbox{ s.t. } \|e^{it(\partial_x^2 + V)} g- 
e^{it\partial_x^2} g_\pm\|_{L^2(\R)}
\overset{t\rightarrow \pm \infty} \rightarrow 0.
\end{equation}
We recall that  \eqref{reedsimon}
has been established for a large class of potentials which satisfy a suitable long-range
decay condition, namely $V$ has to decay of the order $\langle x\rangle^{-1-\epsilon_0}$
with $\epsilon_0>0$
(see  \cite{Ag}, \cite{En}, \cite{RS2}).
\\

Along the paper the following
asymptotic description for free waves will play an important role (see \cite{D} and \cite{RS2}):
\begin{equation}\label{saymptoticbehavior}
\|e^{it\partial_{x}^2} h - \frac {e^{i\frac{x^2}{4t}}}{\sqrt {2i t}} \hat h(\frac {x}{2t})\|_{L^2(\R)}\overset{t\rightarrow \infty}\longrightarrow 0,
\quad \forall\, h\in L^2({\R}),
\end{equation}
where $\hat h(\xi)$ is the (unique) extension of \eqref{fourier} to functions in $L^2(\R)$.
Next we state our result about the long time description of moments of order $s$.
Since in the linear case we can provide
a better result compared with the nonlinear one, we provide two separate statements. \\
\begin{theorem}\label{V=0}
Let $V\in L^\infty(\R)$ be real valued, $\lim_{x\rightarrow \infty} V(x)=0$. Assume moreover
\eqref{boundedreivatives}, \eqref{posit} and \eqref{reedsimon}.
Let $f\in \Sigma_s(\R)$ with $s\geq 1$ an integer,
then 
\begin{equation}\label{momentumpreciseV}
\lim_{t\rightarrow \pm \infty }
\int_{\R} \big(\frac xt\big)^{2s}|e^{it (\partial_x^2 + V)} f|^2 dx= 2^{2s} \|(\sqrt {-\partial_x^2 - V})^s f\|_{L^2}^2  \,.
\end{equation}
\end{theorem}
Notice that \eqref{momentumpreciseV}
is the perturbative counterpart of \eqref{PR2} under a linear potential pertubation.
\\

Finally we switch to the nonlinear case (namely \eqref{NLS} with $\lambda< 0$)
where, exactly as in Theorem \ref{mainnonlin}, we assume the extra assumption \eqref{decay}.
Moreover despite to Theorem \ref{V=0} we get a slightly weaker conclusion, indeed 
on the r.h.s. in \eqref{momentumprecise} we have the classical Sobolev norm of the scattering states $f_\pm$, while  
in the r.h.s. in \eqref{momentumpreciseV} it is involved the pertubed Sobolev norm of the initial datum $f$ himself.

\begin{theorem}\label{main} 
Assume that $V\in L^\infty(\R)$ is real valued and satisfies \eqref{boundedreivatives}, 
\eqref{posit}, \eqref{reedsimon}. Let
$\lambda< 0$ and $k\geq 2$ be an integer. Let $f\in \Sigma_s(\R)$ with $s\geq 1$ an integer, and assume that the unique global solution $u(t,x)
\in \mathcal(\R;H^1(\R))$ to \eqref{NLS}
satisfies \eqref{decay}. Then
$u(t,x)\in {\mathcal C}(\R;H^s(\R))$
and there exist $f_\pm \in H^s(\R)$ such that
\begin{equation}\label{alglin}
\lim_{t\rightarrow \pm \infty }\|u(t,x)-e^{it\partial_{x}^2} f_\pm\|_{H^s(\R)}=0.
\end{equation}
Moreover 
\begin{equation}\label{momentumprecise}
\lim_{t\rightarrow \pm \infty }
\int_{\R} \big(\frac xt\big)^{2s}|u(t,x)|^2 dx= 2^{2s} \int_{\R} \big|\partial_x^s f_\pm\big|^2 dx  \,.
\end{equation}
\end{theorem}
Several remarks in order are listed below.
\begin{remark}\label{carlino}
In the case $V=0$,  Theorem~\ref{main} can be obtained by using the vector field $J=x+2it\partial_x$, whose main property is
$[J,i\partial_t +\partial_x^2]=0$. Indeed one can get the following bound for solutions to \eqref{NLS} with $V\equiv 0$,
$\lambda\leq 0$:
$$\sup_t \|J^s u(t,x)\|_{L^2(\R)}<\infty$$ (see \cite{carles}).
In turn this estimate along with \eqref{alglin} and
the following identity (that follows by developing in power of $t$ the operator $J^s$)
$$\frac {J^s u(t,x)}{t^s}=\frac{x^s}{t^s} u(t,x) + (2i)^s \partial_x^s u(t,x)+o(1) 
\hbox{ as } t\rightarrow \infty$$ implies \eqref{momentumprecise} for 
the translation invariant defocusing NLS. 
\\
However the argument that we propose in this paper in order to prove \eqref{momentumprecise} is more flexible, it is not based on the commuting vector fields technique, and in particular it works in the more general no translation invariant setting. 
\end{remark}
\begin{remark}
As we shall see, 
the proof of \eqref{alglin} is straightforward  under
the assumption \eqref{decay}. 
There is a huge literature devoted to the proof
of \eqref{alglin} without assuming a priori the strong time decay \eqref{decay} on the solution 
neither  the space decay assumption $|x|^s f\in L^2(\R)$ on the initial datum. In this direction we recall 
\cite{N} in the case $V=0$ and $k>2$. 
Different proofs can be found in \cite{PV}, \cite{TV}
\cite{V} via interaction Morawetz estimates.
The critical case $k=2$ and $V=0$ has been studied in \cite{D}. 
Concerning the case of defocusing NLS with a potential $V\neq 0$ and $k>2$,
the scattering in Sobolev spaces has been established in several papers, in between we quote  \cite{BV}, \cite{FV}, 
\cite{L} and the references therein.
\end{remark}
\begin{remark}
It is worth mentioning that the proof of \eqref{momentumprecise}
follows, at least formally, by \eqref{PR2} and \eqref{alglin}. However the argument needed to get \eqref{momentumprecise} is more subtle
since $L^2$ weighted integrals are involved on the l.h.s., 
despite to the convergence established in \eqref{alglin} which occurs in Sobolev space $H^s(\R)$ only. In fact at the best of our knowledge it is not clear whether or not 
the scattering states $f_\pm$ that appear in \eqref{alglin} belong to $\Sigma_s(\R)$, and as a result we do not know if
$u(t,x)-e^{it\partial_{x}^2} f_\pm\in \Sigma_s(\R)$.

\end{remark}

\begin{remark}
Notice that as a consequence of Theorem \ref{main}
one can reconstruct the $H^s(\R)$ norm of the scattering state by using $L^2$ weighted norms in physical space. This is reminiscent of the identities proved
in \cite{VV2}, \cite{VV1} where the $H^\frac 12(\R)$ norm 
of the scattering state is obtained as a limit of
space time integrals on cylinders bounded in physical space and infinite in the time variable.
See also \cite{VV} for a generalization to the nonlinear wave equation.
\end{remark}
\section{Equivalence of Sobolev norms and applications}\label{equivsob}
Along this section we collect some results useful in the sequel.
The first result  concerns the equivalence between classical and perturbed $L^2$ based Sobolev spaces. Recall that we assume $-\partial_x^2 - V\geq 0$.  In such a way the operator 
$\sqrt {-\partial_x^2 - V}$ is meaningful. In particular we can give 
the following definitions of perturbed Sobolev spaces
(resp. homogeneous and inhomogeneous):
$$\|f\|_{\dot H^s_V(\R)}^2=\|(\sqrt {-\partial_x^2 - V})^s f\|_{L^2(\R)}^2,$$
$$\|f\|_{ H^s_V(\R)}^2=\|(\sqrt {-\partial_x^2 - V})^s f\|_{L^2(\R)}^2
+\|f\|_{L^2(\R)}^2.$$

\begin{prop}\label{danmig}
Let  $s>0$ be an  integer and $V$ be a real valued potential 
that satisfies \eqref{posit} and moreover
$\partial_x^{j} V\in L^\infty(\R)$ for $j=0,\cdots ,s$.
Then there exist $c, C>0$ such that:
\begin{equation}\label{equivalence}
c \|f\|_{H^s_V(\R)} \leq \|f\|_{H^s(\R)}\leq C \|f\|_{H^s_V(\R)}.
\end{equation}
\end{prop}
\begin{proof}
We shall prove that
\begin{multline}\label{equivhard}
\|\partial_x ^s f\|_{L^2(\R)}^2+\|f\|_{L^2(\R)}^2
\\
\leq 2\big( (\sqrt {-\partial_x^2 - V})^s f, (\sqrt {-\partial_x^2 - V})^s f\big)_{L^2(\R)}+M \|f\|_{L^2(\R)}^2
\end{multline}
and
\begin{multline}\label{equiveasy}
\big((\sqrt {-\partial_x^2 - V})^s f, (\sqrt {-\partial_x^2 - V})^s f\big)_{L^2(\R)}+\|f\|_{L^2(\R)}^2\\
\leq M \|\partial_x ^s f\|_{L^2(\R)}^2+ M \|f\|_{L^2(\R)}^2
\end{multline}
for a suitable $M>0$ large enough.  

First we notice that for every $j, l\in \{0,\dots, s\}$ such that $j+l<2s$ 
and for every $W\in L^\infty(\R)$ the following holds: 
for every $K>0$ there exists $\epsilon(K)$ such 
that $\epsilon(K)\overset{K\rightarrow \infty} \rightarrow 0$ and moreover
\begin{equation}\label{KcK}
|\int_{\R} W\partial_x^j f (x)\partial_x^l \bar f (x) dx|
\lesssim  K\|f\|_{L^2({\R})}^2 + \epsilon(K)\|\partial_x^s f\|_{L^2({\R})}^2.
\end{equation}
In fact by the Cauchy-Schwartz inequality and interpolation there exist $\theta_1, \theta_2\in [0,1]$
such that
\begin{multline*}
|\int_{\R} W\partial_x^j f(x) \partial_x^l \bar f(x) dx|
\lesssim  \|\partial_x^j f\|_{L^2({\R})}
\|\partial_x^l f\|_{L^2({\R})}
\\\lesssim  \|f\|_{L^2({\R})}^{\theta_1} \|\partial_x^s f\|_{L^2({\R})}^{1-\theta_1}
\|f\|_{L^2(\R)}^{\theta_2} \|\partial_x^s f\|_{L^2({\R})}^{1-\theta_2}
\end{multline*}
then we conclude by the Young inequality and by noticing that, since $j+l<2s$ then
necessarily $\min\{\theta_1, \theta_2\}<1$.\\

Next, we sketch the proof of  \eqref{equiveasy}. 
By developing the operator $(\sqrt {-\partial_x^2 - V})^s$ 
and by integration by parts  we get
\begin{multline}\label{equivint}
\big((\sqrt {-\partial_x^2 - V})^s f, (\sqrt {-\partial_x^2 - V})^s f\big )_{L^2({\R})}=
\int_{\R} (-\partial_x^2 - V)^{s} f \bar f dx\\= \int_{\R} |\partial_x^s f|^2 dx
+ \sum_{\substack{l,j\in {0, \cdots, s} \\j+l<2s}} \int_{\R} W_{j,l}\partial_x^j f \partial_x^l \bar f dx\,
\end{multline}
where $W_{j,l}$ are linear combination of product of derivatives of $V$ of order at most $s$.
In particular by assumption we have $W_{j,l}\in L^\infty$. 
It is easy now to get \eqref{equiveasy} by recalling \eqref{KcK}.

Concerning the proof of   \eqref{equivhard} we notice that again by \eqref{equivint}
and \eqref{KcK}
we get the existence of $C>0$ such that
$$\big((\sqrt {-\partial_x^2 - V})^s f, (\sqrt {-\partial_x^2 - V})^s f\big )_{L^2({\R})}$$
 $$\geq \int_{\R} |\partial_x^s f|^2 dx - \frac 12 \|\partial_x^s f\|_{L^2({\R})}^2
-C \|f\|_{L^2({\R})}^2$$
and we easily conclude.
\end{proof}
As an application of Proposition \ref{danmig} we up-grade the $L^2$ asymptotic completeness assumption to $H^s$ for any $s$.
\begin{prop}\label{upgrades} Let $s\geq 0$ be an integer and $V$ be a potential that satisfies
the same assumptions as in Proposition~\ref{danmig}
and the condition \eqref{reedsimon}.
Then the following is true:
\begin{equation}\label{reedsimons}
\forall g\in H^s(\R) \quad \exists g_\pm \in H^s(\R) \hbox{ s.t. } \|e^{it(\partial_x^2 + V)} g- 
e^{it\partial_x^2} g_\pm\|_{H^s(\R)}
\overset{t\rightarrow \pm \infty} \longrightarrow 0.
\end{equation}
\end{prop} 
\begin{proof}
Notice that if $s=0$ then \eqref{reedsimons} it is equivalent to \eqref{reedsimon}.
In the case $s>0$
notice that $\sup_t \|e^{-it\partial_x^2}\circ e^{it(\partial_x^2 + V)} g\|_{H^s(\R)}<\infty$
due to \eqref{equivalence} and to the fact that 
$e^{it(\partial_x^2 + V)}$ and $e^{-it\partial_x^2}$ are isometries respectively in 
$H^s_V(\R)$ and $H^s(\R)$.
Then it is easy to deduce by \eqref{reedsimon} that the functions $g_\pm$, that in principle belong to $L^2(\R)$,
indeed belong to $H^s(\R)$. Next notice that
if moreover we assume $g\in H^{s+1}(\R)$ then, following the argument as above, there exist
$g_\pm\in H^{s+1}(\R)$ such that
$\|e^{it(\partial_x^2 + V)} g- 
e^{it\partial_x^2} g_\pm\|_{L^2(\R)}
\overset{t\rightarrow \pm \infty} \longrightarrow 0$.
By combining this fact with the uniform boundedness of $e^{it(\partial_x^2 + V)} g$
and $e^{it\partial_x^2} g_\pm$ in $H^{s+1}(\R)$ we conclude by interpolation that 
$\|e^{it(\partial_x^2 + V)} g- 
e^{it\partial_x^2} g_\pm\|_{H^s(\R)}
\overset{t\rightarrow \pm \infty} \longrightarrow 0$.
Summarizing we have proved the following fact:
\begin{equation}\label{reedsimonsdense}
\forall g\in H^{s+1}(\R) \quad \exists g_\pm \in H^{s+1}({\R}) \hbox{ s.t. } \|e^{it(\partial_x^2 + V)} g- 
e^{it\partial_x^2} g_\pm\|_{H^s(\R)}
\overset{t\rightarrow \pm \infty} \longrightarrow 0.
\end{equation}
In particular we have proved  that the one parameter family of operators
$$e^{-it\partial_x^2}\circ e^{it(\partial_x^2 + V)}\in {\mathcal L}(H^s({\R}), H^s({\R}))$$
is uniformly bounded and converges pointwisely on a dense subspace of $H^s({\R})$
(namely $H^{s+1}(\R)$)
as $t\rightarrow \pm \infty$.
By a straightforward density argument we deduce that the same property is true on $H^s({\R})$ and we conclude.
\end{proof}
By the proof of Proposition \ref{danmig} one can get the following result which
will be the key point to get
\eqref{momentumpreciseV}.
\begin{prop}\label{thmcorg}
Let $s\geq 0$ be an integer and $V$ a real valued potential that satisfies
$\partial_x^{j} V\in L^\infty(\R)$ and $\lim_{x\rightarrow \infty} 
\partial_x^{j} V=0$ for $j=0,\cdots ,s$. Assume moreover
\eqref{reedsimon} and
for every $g\in H^s(\R)$ let $g_\pm\in L^2(\R)$ be given by \eqref{reedsimon}.
Then necessarily $g_\pm \in H^s(\R)$ and moreover
\begin{equation}\label{thmcrogg}
\int_{\R} \big|\partial_x^s g_\pm\big|^2 dx= \|(\sqrt{-\partial_x^2 - V})^sg\|_{L^2(\R)}^2.
\end{equation}
\end{prop}
\begin{proof}
We claim that
\begin{equation}\label{claimsobV}
\|e^{it (\partial_x^2 + V)} g\|_{\dot H^s_V(\R)}^2=\|e^{it (\partial_x^2 + V)} g\|_{\dot H^s(\R)}^2+ o(1),
\quad { as }  \quad t\rightarrow \pm \infty.
\end{equation}
On the other hand by using Proposition \ref{upgrades}
we get 
\begin{equation}\label{asymptmer}\|e^{it (\partial_x^2 + V)} g - e^{it \partial_x^2} g_\pm\|_{\dot H^s(\R)}\overset{t\pm \infty}
\rightarrow 0\end{equation}
and we also have
\begin{equation}\label{isomV=0}
\|e^{it \partial_x^2} g_\pm\|_{\dot H^s(\R)}=\|g_\pm\|_{\dot H^s(\R)},
\quad \|g\|_{\dot H^s_V(\R)}=\|e^{it (\partial_x^2 + V)} g\|_{\dot H^s_V(\R)}
\end{equation}
by the fact that 
$e^{it (\partial_x^2 + V)}$ is an isometry on $\dot H^s_V(\R)$. It is now easy to conclude \eqref{thmcrogg}.
Next we prove \eqref{claimsobV}. In order to do that first notice that
\begin{equation}\label{impdeca}
\forall\, \epsilon>0 \quad \exists\, \delta>0 \hbox{ s.t. }
\limsup_{t\rightarrow \pm \infty}  \|\partial_x^j v(t)\|_{L^2(|x|<\delta t)}<\epsilon,
\quad \forall j=0,\dots,s
\end{equation}
where $v(t,x)=e^{it(\partial_x^2 + V)} g$. 
Indeed, by combining 
\eqref{reedsimons} with \eqref{saymptoticbehavior} and the change of variable formula, then we get
\begin{multline*}\int_{|x|<\delta t} |\partial_x^j v (t,x)|^2 dx=\int_{|x|<\delta t} |\partial_x^j u_\pm (t,x)|^2 dx
+o(1)\\=  \int_{|x|<\delta/2} 
|{\mathcal F} ({\partial_x^j g_\pm})|^2 dx +o(1), \hbox{ as } \quad t\rightarrow \pm \infty
\end{multline*}
where $u_\pm(t,x)=e^{it\partial_x^2} g_\pm$. We conclude since
we have 
$$
\lim_{\delta\rightarrow 0}
\int_{|x|<\delta/2} 
|{\mathcal F} ({\partial_x^j g_\pm})|^2 dx=0\,.
$$
Next, notice that by \eqref{equivint}, in order to conclude \eqref{claimsobV}, it is sufficient to show 
\begin{equation}\label{crossterm}\int_{\R} W_{j,l} \partial_x^j v(t,x) \partial_x^l \bar v(t,x) dx\overset{t \pm \infty} \rightarrow 0, \quad \forall j, l\in \{0,\dots, s\},\quad j+l<2s
\end{equation} 
where $\lim_{|x|\rightarrow \infty} W_{j,l}=0$, since $W_{j,l}$ is a linear combination of product of 
functions that decay by assumption.
 Notice that by the equivalence of Sobolev spaces (see Proposition
 \ref{danmig}) and by using 
 the fact that $e^{it(\partial_x^2 + V)}$  are isometries on $H^s_V(\R)$,
 we have
\begin{equation}\label{HsHsV}\sup_t \|v(t,\cdot)\|_{H^s(\R)}<\infty.\end{equation}
Moreover we have (if $\delta$ is given by \eqref{impdeca} for a fixed $\epsilon>0$) 
 \begin{multline*}|\int_{\R} W_{j,l} \partial_x^j v(t,x) \partial_x^l \bar v(t,x) dx|
 \\\leq \int_{|x|<\delta t} |W_{j,l} \partial_x^j v(t,x) \partial_x^l \bar v(t,x)| dx+\int_{|x|>\delta t} |W_{j,l} \partial_x^j v(t,x) \partial_x^l \bar v(t,x)| dx\\
\leq \|W_{j,l}\|_{L^\infty(\R)} 
\sqrt{(\int_{|x|<\delta t} |\partial_x^j v(t,x)|^2 dx)}\|v(t,x)\|_{H^s(\R)} \\+
\|W_{j,l}\|_{L^\infty(|x|>\delta t)} \|v(t,x)\|_{H^s(\R)}^2.\end{multline*}
We conclude \eqref{crossterm} by combining \eqref{HsHsV}, 
\eqref{impdeca} and $\lim_{|x|\rightarrow \infty} W_{j,l}=0$.  
 \end{proof}
\section{Proof of Theorems \ref{mainlin} and \ref{mainnonlin}}
The main of this section is to prove the following  upper bound
\begin{equation}\label{upperboundt2s}
\int_{\R} x^{2s}|u(t,x)|^2 dx\lesssim \langle t\rangle ^{2s}\,.
\end{equation}
under the assumptions of Theorems \ref{mainlin} and \ref{mainnonlin}.
We provide one proof that works in both linear and nonlinear case (namely $\lambda=0$ and $\lambda<0$) , except that in the case $\lambda< 0$ we get an extra term that we shall treat by using the assumption \eqref{decay} done along the statement
of Theorem \ref{mainnonlin}. However as we shall see the assumption \eqref{decay}
is not needed in the linear case $\lambda=0$.
\\

We shall need the following proposition.

\begin{prop}\label{gron}
Let $H(t)\in {\mathcal C}^0([0, \infty); \R)\cap L^1((0,\infty);\R)$, $\alpha_i>0$ and $\beta_i\in [0, 1)$ for $i=1,\dots,m$, $C>0$ be given.
Let $F(t)\in {\mathcal C}^1((0, \infty);\R)\cap {\mathcal C}^0
([0, \infty);\R)$ be any function such that:
\begin{enumerate}
\item $F(t)\geq 0, \quad \forall t>0$;
\item $F(0)\leq C$;
\item the following inequality holds 
\begin{equation}\label{gronnew}
|\frac d{dt}F(t)|\leq C \sum_{i=1}^m \langle t\rangle^{\alpha_i} (F(t))^{\beta_i}+H(t) F(t),
\quad \forall t>0.
\end{equation} 
\end{enumerate}
Then there exists $K=K(\alpha_i, \beta_i, H, C)>0$ such that
$$F(t)\leq \max_{i=1,\dots, m} (K+Km
\langle t \rangle^{\alpha_i+1})^\frac 1{1-\beta_i}, \quad \forall t>0.$$
\end{prop}

\begin{proof}
As a first step we show a general fact that we shall use in the sequel.
Let $G(t)\geq 0$ be a generic 
non negative function, then we have the following implication:
\begin{multline}\label{gronnewnewR}
\exists C'>0, R_i>0, \beta_i\in [0,1) \hbox{ for } i=1,\dots, m \quad \hbox{ s. t. }
\\
G(t)\leq C'+ C'\sum_{i=1}^m R_i (G(t))^{\beta_i},
\\\Longrightarrow G(t)\leq \max_{i=1,\dots, m}
(\max\{1,C'\}+\max \{1, C'\} m R_i)^\frac 1{1-\beta_i}.\end{multline}
Notice that it is not restrictive to assume $C'\geq 1$, in fact if it is not the case
then the assumption implies
$G(t)\leq 1+ \sum_{i=1}^m R_i (G(t))^{\beta_i}$ and the conclusion becomes
$G(t)\leq \max_{i=1,\dots, m}
(1+m R_i)^\frac 1{1-\beta_i}.$
We introduce the function
$$(0, \infty)\ni y\rightarrow H (y)=y- C'-C' \sum_{i=1}^m R_i y^{\beta_i}.$$
Notice that 
\begin{equation}\label{appen1}\frac {d^2}{d y^2} H(y)\geq 0, \quad H(0)\leq 0, \quad
\lim_{y\rightarrow \infty} H(y)=\infty
\end{equation}
then it is easy to show that 
\begin{equation}\label{appen2}\exists! y_0\in (0, \infty) \hbox{ s.t. } H(y_0)=0.
\end{equation}
Hence by \eqref{gronnewnewR} we get
$$G(t)\in \{y>0| H(y)\leq 0\}\equiv [0,y_0].$$
We claim that
$$H(z_0)>0, \hbox{ where } z_0=\max_{i=1,\dots, m}
\{(C'+C'mR_i)^\frac 1{1-\beta_i}\}$$
and hence 
$z_0>y_0$, which in turn implies
$G(t)\in (0,z_0)$ and we conclude \eqref{gronnewnewR}.
In order to show $H(z_0)>0$ it is sufficient to notice that if
we define $$z_{0,i}=(C'+C'm R_i)^\frac 1{1-\beta_i}$$ 
then (since we are assuming $C^{'}\geq 1$) by direct inspection
$$\frac {z_{0,i}}m - \frac{C'}m - C' R_i  z_{0,i}^{\beta_i}>0.$$
Moreover
the function
$y\rightarrow \frac {y}m - \frac{C'}m - C' R_i y^{\beta_i}$
is increasing in the region where it is positive (it follows from the fact that this function has similar properties to the ones satisfied by $H(y)$, see \eqref{appen1})
and it implies 
$$\frac {z_0}m - \frac{C'}m - C' R_i z_0^{\beta_i}>0, \hbox{ where } 
z_0=\max\{z_{0,i}, i=1,...,m\}.$$
Summarizing we get
\begin{equation*}H(z_0)= z_0- C' - C' \sum_{i=1}^m R_i z_0^{\beta_i}
=
\sum_{i=1}^m \big( \frac {z_0}m - \frac {C'}m - C' R_i  z_0^{\beta_i}
\big) 
>0.\end{equation*}
Next we come back to our original function $F(t)$, and we 
claim that the following inequality occurs:
\begin{equation}\label{gronnewnew}
\tilde F(t)\leq C''+ C''\sum_{i=1}^m \langle t \rangle^{\alpha_i+1} (\tilde F(t))^{\beta_i}\end{equation}
with a constant $C''=C''(\alpha_i, \beta_i, H, C)>1$, where
$\tilde F(t)=\sup_{s\in (0, t)} F(s)$.
Once this is established then 
by \eqref{gronnewnewR}, where we choose $G(t)=\tilde F(t)$,
we get $$\tilde F(t)\leq \max_{i=1,\dots, m}
(\max\{1,C''\}+\max\{1,C''\} m\langle t\rangle^{\alpha_i+1})^\frac 1{1-\beta_i}$$ and we conclude.
Next we show that \eqref{gronnew} implies \eqref{gronnewnew} and the proof will
be complete.
Since $H(t)\in L^1((0, \infty);\R)$ we can fix $t_0>1$ large enough in such a way that
$\int_{t_0}^\infty H(\tau) d\tau<\frac 12$, then by \eqref{gronnew} 
$$\sup_{s\in [t_0, t]}F(s)\leq F(t_0)+ C 
 \sum_{i=1}^m  (\sup_{s\in [t_0, t]}F(s))^{\beta_i} \int_{t_0}^t \langle \tau \rangle^{\alpha_i} 
 d\tau  + \frac 12 (  \sup_{s\in [t_0, t]}F(s))$$
and hence (since we have choosen $t_0>1$ and hence $\langle \tau \rangle\leq \sqrt 2 \tau$ for $\tau>t_0$)
\begin{equation}\label{tzero}\sup_{s\in [t_0, t]}F(s)\leq 2F(t_0) + 
2 C \sum_{i=1}^m \frac{(\sqrt 2)^{\alpha_i} t^{\alpha_i+1} }{\alpha_i+1}  (\sup_{s\in [t_0, t]}F(s))^{\beta_i}
, \quad \forall t>t_0.
\end{equation}
Notice that since $H\in L^1(0, \infty)$ we can choose, along with the point $t_0$ 
fixed above, other points such that $t_0>t_1>t_2\dots >t_k=0$ such that
\begin{equation}\label{pratfin}\int_{t_{j+1}}^{t_{j}} H(\tau) d\tau<\frac 12, \quad  
\int_{t_{j+1}}^{t_{j}} \langle \tau\rangle^{\alpha_i} d\tau< 1.\end{equation}
Arguing as along the proof of \eqref{tzero} and by using \eqref{pratfin} we get:
\begin{equation}\label{tjtj+1}\sup_{s\in [t_{j+1}, t]}F(s)\leq 2F(t_{j+1}) + 
2 C \sum_{i=1}^m (\sup_{s\in [t_{j+1}, t]}F(s))^{\beta_i}, \quad \forall t\in [t_{j+1}, t_j].
\end{equation}
It is now easy to deduce from \eqref{tzero} and \eqref{tjtj+1} that 
$$\sup_{s\in [0, t]} F(s)\leq 2 k \max_{j=0,\dots, k} F(t_j) +
2 C k \sum_{i=1}^m (1+\frac{(\sqrt 2)^{\alpha_i} t^{\alpha_i+1} }{\alpha_i+1})    (\sup_{s\in [0, t]} F(s))^{\beta_i},
\quad \forall t>0
$$
and from this we get \eqref{gronnewnew}, provided that we show
$\max_{j=0,\dots, k} F(t_j)< C^{'''}$ where $C^{'''}=C^{'''}(C, \alpha_i, \beta_i, H)>0$.
In order to do that we go back to \eqref{tjtj+1}, and notice that 
by using again \eqref{gronnewnewR}, where we choose $G(t)=G_j(t)=\sup_{s\in [t_{j+1}, t]} F(s)$,
we get
\begin{multline*}F(t_j)\leq G_j(t_j)\\\leq \max_{i=1,\dots, m}
\big (\max\{1, 2F(t_{j+1}), 2C\}+\max\{1, 2F(t_{j+1}), 2C\} m
\big )^\frac 1{1-\beta_i}\end{multline*}
and hence $F(t_j)$ can be controlled by $F(t_{j+1})$. If we iterate $k$ times this construction
we control $F(t_j)$ by a function of $F(t_k)=F(0)<C$.


\end{proof}

We shall also need the following proposition to deal with the nonlinear case $\lambda< 0$.

\begin{prop}
Assume $V\in L^\infty$ is real valued and $\partial_x^j V\in L^\infty$ for $j=1,\dots, s$. Let $u(t,x)$ be as in Theorem \ref{main}, then we have
\begin{equation}\label{Hs}\sup_{t}\|u(t,x)\|_{H^s({\R})}<\infty.\end{equation}
\end{prop}
\begin{proof} We estimate the $\sup$ for $t\in (0, \infty)$, the estimate for $t\in (-\infty, 0)$ is similar.
From now on we shall use without any further comments the equivalence of norms \eqref{equivalence}.
By the integral formulation of \eqref{NLS} we get  for every $T\geq 0$:
\begin{multline*}
\|u(t,x)\|_{H^s({\R})}\lesssim \|u(T,x)\|_{H^s({\R})}+\|\int_T^t  e^{i(t-\tau)(\partial_x^2 + V)} (u|u|^{2k}) d\tau\|_{H^s({\R})}
\\
\lesssim \|u(T,x)\|_{H^s({\R})} + \int_T^\infty \|u|u|^{2k}\|_{H^s({\R})} d\tau
\\
\lesssim \|u(T,x)\|_{H^s({\R})}+ \sup_{t>T} \|u(t,x)\|_{H^s({\R})} 
\int_T^\infty \frac 1{\langle \tau \rangle^{k}} d\tau
\end{multline*}
where we have used \eqref{decay} at the last step.
Therefore one can conclude the following estimate $\sup_{t>T}\|u(t,x)\|_{H^s({\R})}<\infty$ provided that we choose $T>0$ large enough.
The proof of \eqref{Hs} is now straightforward.

\end{proof}

\noindent {\em Proof of \eqref{upperboundt2s}.} Next we shall prove
the following generalization of  \eqref{upperboundt2s}:
\begin{equation}\label{induct1}
\int_{\R} x^{2s-2h}|\partial_x^h u(t,x)|^2 dx\lesssim 
\langle t\rangle ^{2s-2h}, \quad \forall h=0,\dots , s.
\end{equation}
We shall work first by induction on $s$ and then by backward induction on $h$.
Hence we assume that \eqref{induct1} has been established
for all $\bar s<s$ and for all $h=0,\dots, \bar s$. Notice that this fact implies
\begin{equation}\label{induct1simple}
\int_{\R} x^{2s-2k}|\partial_x^j u(t,x)|^2 dx\lesssim 
\langle t\rangle ^{2s-2k}, \quad \forall 0\leq j<k\leq s.
\end{equation}
Based on \eqref{induct1simple} we shall establish \eqref{induct1}.

For $h=s$ we have that \eqref{induct1} follows by $\sup_t \int |\partial_x^s u(t,x)|^2 dx<\infty$.
This fact follows in the linear case ($\lambda=0$) from the conservation of $\|e^{it(\partial_x^2 + V)} f\|_{H^s_V(\R)}$ in conjunction with 
Proposition \ref{danmig}.
In the nonlinear case ($\lambda<0$) it follows by \eqref{Hs}.\\

{\bf Next we establish \eqref{induct1} for 
$h=1, \cdots, s-1$.}
\\

We use backward induction on $h$. Assume that \eqref{induct1} has been proved for 
$h=\bar h,\cdots, s$ with $\bar h\geq 2$, then we prove that it is true for $h=\bar h-1$.
We point out that in principle it is unclear that the quantity
$\int_{\R} x^{2s-2\bar h +2}|\partial_x^{\bar h-1} u(t,x)|^2
dx$ is finite. For this reason we introduce a family of cut-off functions and at the end we pass to the limit as $\epsilon\rightarrow 0$ by getting uniform bounds. We define 
\begin{multline*}\varphi_\epsilon(x)= \epsilon^{-s+\bar h -1}\psi(\epsilon x)
\hbox{ where }\psi(x)=\eta(x) x^{s-\bar h +1} \hbox{ and }\\
\eta\in C^\infty_0(\R), \quad \eta(x)=1 \quad \forall x\in (-1, 1),
\quad \eta(x)=0 \quad \forall |x|>2.\end{multline*}
Next we compute
\begin{equation*}\frac{d}{dt} \int_{\R} \varphi_\epsilon^2|\partial_x^{\bar h-1} u(t,x)|^2
dx=2 \Re \int_{\R} \varphi_\epsilon^2\partial_x^{\bar h-1} \partial_t u(t,x) \partial_x^{\bar h-1} \bar u(t,x)
dx 
\end{equation*}
and by using the equation solved by $u(t,x)$ we have
\begin{multline*}
\dots=2 \Re \int_{\R}  \varphi_\epsilon^2 \partial_x^{\bar h-1} i \partial_{x}^2 u(t,x) \partial_x^{\bar h-1} \bar u(t,x)
dx \\+ 2\lambda  \Re \int_{\R}  \varphi_\epsilon^2 \partial_x^{\bar h-1} i (u(t,x)
|u(t,x)|^{2k}) \partial_x^{\bar h-1} \bar u(t,x)
dx\\+2 \Re \int_{\R}  \varphi_\epsilon^2 \partial_x^{\bar h-1} i (Vu(t,x)) \partial_x^{\bar h-1} \bar u(t,x)
dx\\
=-2 \Im \int_{\R}  \varphi_\epsilon^2 \partial_x^{\bar h+1} u(t,x) \partial_x^{\bar h-1} \bar u(t,x)
dx \\- 2 \lambda \Im \int_{\R}  \varphi_\epsilon^2 \partial_x^{\bar h-1}  (u(t,x)|u(t,x)|^{2k}) \partial_x^{\bar h-1} \bar u(t,x) dx\\-2 \Im \int_{\R}  \varphi_\epsilon^2\partial_x^{\bar h-1} (Vu(t,x)) \partial_x^{\bar h-1} \bar u(t,x)
dx
=I_\epsilon+II_\epsilon+III_\epsilon.\end{multline*}
By integration by parts we get
\begin{multline}\label{gronc}|I_\epsilon|=4 |\Im \int_{\R}   \varphi_\epsilon \partial_x \varphi_\epsilon \partial_x^{\bar h} u(t,x) \partial_x^{\bar h-1} \bar u(t,x)
dx| \\\lesssim \sqrt{\int_{\R} \varphi_\epsilon^2 |\partial_x^{\bar h-1} u(t,x)|^2
dx } \sqrt{\int_{\R} (\partial_x \varphi_\epsilon)^2 |\partial_x^{\bar h} u(t,x)|^2
dx }.
\end{multline}
Notice that
$$\partial_x \varphi_\epsilon(x)=\epsilon^{-s+\bar h}
[\partial_x \eta(\epsilon x) (\epsilon x)^{s-\bar h +1} 
+ (s-\bar h +1) (\epsilon x)^{s-\bar h} \eta(\epsilon x) ]$$
and hence, due to the support of $\eta, \partial_x \eta $ we get
$$|\partial_x \varphi_\epsilon|(x)\lesssim |x|^{s-\bar h}.$$
Then from \eqref{gronc} we get
\begin{multline}\label{Iepsilon}
|I_\epsilon| \lesssim 
\sqrt{\int_{\R} \varphi_\epsilon^2 |\partial_x^{\bar h-1} u(t,x)|^2
dx } \sqrt{\int_{\R} x^{2s-2\bar h} |\partial_x^{\bar h} u(t,x)|^2
dx }
\\\lesssim \langle t\rangle^{s-\bar h}
 \sqrt{\int_{\R} \varphi_\epsilon^2|\partial_x^{\bar h-1} u(t,x)|^2
dx }\,,
\end{multline} 
where we have used the inductive assumption that \eqref{induct1} is true for $h=\bar h$.

Concerning the estimate of $II_\epsilon$ notice that
\begin{multline}\label{gobak} |II_\epsilon|=2\lambda |\Im \int_{\R}  \varphi_\epsilon^2 \partial_x^{\bar h-1}  (u(t)|u(t)|^{2k}) \partial_x^{\bar h-1} \bar u(t)
dx|
\\\lesssim \|u(t)\|_{L^\infty({\R})}^{2k}
\int_{\R}  \varphi_\epsilon^2(x) |\partial_x^{\bar h-1} u(t)|^2 dx
\\+ 
\sum_{\substack{j_1,\dots, j_{2k+1}\in \N\\\sum_{l=1}^{2k+1} j_l=\bar h-1\\
j_1\leq \bar h -2
}}\sqrt{\int_{\R}  \varphi_\epsilon^2 |\partial_x^{\bar h-1} u(t,x)|^2 dx}
\sqrt{\int_{\R}  \varphi_\epsilon^2 |\partial_x^{j_1} u(t,x)|^2 dx} 
\prod_{i=2}^{2k+1} \|\partial_x^{j_i} u(t,x)\|_{L^\infty({\R})}.\end{multline}
Moreover by Sobolev embedding $H^1({\R})\subset L^\infty
({\R})$ and interpolation we get:
\begin{multline*}
\|\partial_x^{j_i} u(t,x)\|_{L^\infty({\R})}
\lesssim \|u(t,x)\|_{L^\infty(\R)}^{1-\frac{j_i}{s-1}}
\|\partial_x^{s-1}u(t,x)\|_{L^\infty({\R})}^{\frac{j_i}{s-1}}
\\\lesssim 
\|u(t,x)\|_{L^\infty({\R})}^{1-\frac{j_i}{s-1}}
\|u(t,x)\|_{H^s({\R})}^{\frac{j_i}{s-1}}, \quad \forall i=2,\dots, 2k+1.
\end{multline*}
Hence we get 
$$\prod_{i=2}^{2k+1} \|\partial^{j_i}_x u(t,x)\|_{L^\infty({\R})}\lesssim 
\|u(t,x)\|_{L^\infty({\R})}^{2k-\frac{\sum_{i=2}^{2k+1} j_i}{s-1}}
\lesssim \frac 1{\langle t\rangle^{1+\gamma}}, \quad \hbox { for some } \gamma>0$$
where we used (due to the contratint imposed on $j_l$) $\sum_{i=2}^{2k+1} j_i< \bar h-1\leq s-1$, \eqref{decay}
and the fact that $\sup_t \|u(t,x)\|_{H^s(\R)}<\infty$. On the other hand
we have  by the induction assumption \eqref{induct1simple}
$$\int_{\R}  \varphi_\epsilon^{2} |\partial_x^{j_1} u(t,x)|^2 dx
\lesssim \int_{\R}  |x|^{2s-2\bar h +2} |\partial_x^{j_1} u(t,x)|^2 dx \lesssim 
\langle t \rangle^{2s-2\bar h +2}.$$
Summarizing, going back to \eqref{gobak} and by recalling again \eqref{decay}, we get
\begin{equation}\label{IIepsilon}
|II_{\epsilon}|\lesssim \langle t \rangle^{-k} \int_{\R}  \varphi_\epsilon^2 |\partial_x^{\bar h-1} u(t,x)|^2
dx+\langle t \rangle^{s-\bar h-\gamma}
 \sqrt{\int_{\R}  \varphi_\epsilon^2 |\partial_x^{\bar h-1} u(t,x)|^2 dx}\,.
 \end{equation}

Next we estimate $III_\epsilon$, and we notice that by using the Leibniz we get:
\begin{multline}\label{IIIepsilon}
|III_\epsilon|= 2|\Im \int_{\R}  \varphi_\epsilon^{2} \partial_x^{\bar h-1} (Vu(t,x)) \partial_x^{\bar h-1} \bar u(t,x)
dx|\\
\lesssim \sum_{0\leq j\leq \bar h -2}  \int_{\R}  \varphi_\epsilon^{2}
|\partial_x^{\bar h-1-j} V| |\partial_x^j u(t,x)| |\partial_x^{\bar h-1} \bar u(t,x)|
dx
\\\lesssim  \sum_{0\leq j\leq \bar h -2} \sqrt{\int_{\R} |\partial_x^{\bar h-1-j} V|^2 
\varphi_\epsilon^{2}|\partial_x^{j} u(t,x)|^2
dx}
 \sqrt{\int_{\R} \varphi_\epsilon^{2}
|\partial_x^{\bar h-1} u(t,x)|^2
dx} \\\lesssim  
\langle t\rangle^{s-\bar h}\sqrt{\int_{\R} \varphi_\epsilon^{2}|\partial_x^{\bar h-1} u(t,x)|^2
dx}
\end{multline} where at the last step we used the inductive assumption \eqref{induct1simple}
in conjunction with the fact that by \eqref{boundedreivatives} we have the bound
$$|\partial_x^{\bar h-1-j} V|^2 \varphi_\epsilon^{2}\lesssim |x|^{2s-2\bar h}.$$ 
Summarizing we get
\begin{equation}\label{sh1}|\frac d{dt} g_\epsilon(t)|\lesssim \langle t\rangle ^{s-\bar h} \sqrt{g_\epsilon(t)} +
\langle t \rangle^{-k} g_\epsilon(t)\end{equation}
where $g_\epsilon (t)=\int_{\R} \varphi_\epsilon^2|\partial_x^{\bar h-1} u(t,x)|^2
dx$
which in turn by Proposition \ref{gron}, where we choose $F(t)=g_\epsilon(t)$, 
implies
$$g_\epsilon(t)\lesssim \langle t \rangle^{2s-2\bar h +2}$$
with a constant uniform w.r.t. to $\epsilon>0$ and $t$. 
Then we conclude
the desired estimate
$$\int_{\R} x^{2s-2\bar h +2}|\partial_x^{\bar h-1} u(t,x)|^2
dx\lesssim \langle t \rangle^{2s-2\bar h +2}$$
by passing to the limit as $\epsilon\rightarrow 0$.
\\
\\
{\bf Finally we establish \eqref{induct1} for $h=0$.} 
\\
\\
Since a priori it is unclear that $u(t,x)\in \Sigma_s(\R)$, 
we introduce
\begin{multline*}\varphi_\epsilon(x)= \epsilon^{-s}\psi(\epsilon x)
\hbox{ where }\psi(x)=\eta(x) x^{s} \hbox{ and }\\
\eta\in C^\infty_0(\R), \quad \eta(x)=1 \quad \forall x\in (-1, 1),
\quad \eta(x)=0 \quad \forall |x|>2\end{multline*}
and we compute 

\begin{multline*}\big |\frac d{dt} \int_{\R} \varphi_\epsilon^2 |u(t,x)|^2 dx\big |= 2|\Re \int 
\varphi_\epsilon^2 \partial_t u(t,x) \bar u(t,x)|
=2 |\Im \int_{\R} \varphi_\epsilon^2
\partial_{x}^2 u(t,x) \bar u(t,x) dx|
\\=4 |\Im \int_{\R}  \varphi_\epsilon \partial_x \varphi_\epsilon\partial_{x} u(t,x) 
\bar u(t,x) dx|\\\leq 4 \sqrt{\int_{\R} \varphi_\epsilon^2 |u(t,x)|^2 dx} \sqrt{\int_{\R} (\partial_x \varphi_\epsilon)^2 |\partial_x u(t,x)|^2 dx}. 
\end{multline*}
Next notice that
$$\partial_x \varphi_\epsilon(x)=\epsilon^{-s+1}\partial_x \eta(\epsilon x) (\epsilon x)^s+ s 
\epsilon^{-s+1} \eta(\epsilon x)
(\epsilon x)^{s-1} $$
and hence due to the properties of $\eta$ and $\partial_x \eta$ we get
$$|\partial_x \varphi_\epsilon(x)| \lesssim |x|^{s-1}$$
By combining this estimate with the fact that we have established in the previous steps 
\eqref{induct1} for $h=1$, we get
$$|\frac d{dt} \int_{\R} \varphi_\epsilon^2  |u(t,x)|^2 dx|\lesssim 
\langle t\rangle^{s-1} \sqrt{\int_{\R} 
\varphi_\epsilon^2  |u(t,x)|^2 dx}.
$$
By Proposition \ref{gron}
where we choose $F(t)=g_\epsilon(t)=\int_{\R} \varphi_\epsilon^2 |u(t,x)|^2 dx$,
we conclude 
$$g_\epsilon(t)\lesssim \langle t \rangle^{2s}$$ for every $\epsilon>0$ and $t$ 
with a constant uniform w.r.t. $\epsilon, t$.
Then by passing to the limit as $\epsilon\rightarrow 0$ we get the desired bound
$$\int_{\R} x^{2s}  |u(t,x)|^2 dx\lesssim \langle t\rangle ^{2s}.$$



\section{Proof of Theorem \ref{V=0}}\label{mainsection}

First notice that by Proposition \ref{upgrades} there exist
$f_\pm\in H^s(\R)$ such that 
\begin{equation}\label{partdeb}\|e^{it(\partial_x^2 + V)} f- 
e^{it\partial_x^2} f_\pm\|_{H^s(\R)}
\overset{t\rightarrow \pm \infty} \longrightarrow 0.\end{equation}
The same argument to get \eqref{momentumprecise} in the nonlinear case $\lambda<0$ (see section \ref{dana} below),
can be adapted to prove in the linear setting
\begin{equation}\label{riodanfe}
\lim_{t\rightarrow \pm \infty }
\int_{\R} \big(\frac xt\big)^{2s}|e^{it(\partial_x^2 + V)} f|^2 dx= 2^{2s} \int_{\R} \big|\partial_x^s f_\pm\big|^2 dx  \,.
\end{equation}

Then we conclude
\eqref{momentumpreciseV} simply by recalling
\eqref{thmcrogg}.

\section{Proof of Theorem \ref{main}}\label{dana}


First we prove \eqref{alglin}. Recall that we have already proved (see \eqref{Hs})
\begin{equation}\label{Hsprime}\sup_{t}\|u(t,x)\|_{H^s({\R})}<\infty.\end{equation}
Next notice that by the integral formulation of \eqref{NLS} and
by \eqref{equivalence} we get
\begin{multline*}\|e^{-it_2(\partial_x^2 + V)} u(t_2,x)-e^{-it_1(\partial_x^2 + V)} u(t_1,x)\|_{H^s({\R})}
\\=\|\int_{t_1}^{t_2} e^{-i\tau(\partial_x^2 + V)}
(u|u|^{2k}) d\tau\|_{H^s({\R})}\lesssim \int_{t_1}^{t_2}
\| u|u|^{2k}\|_{H^s(\R)} d\tau\\\lesssim \int_{t_1}^{t_2} \|u(\tau,x)\|_{H^s(\R)} \|u
(\tau,x)\|_{L^\infty(\R)}^{2k}d\tau
\lesssim \int_{t_1}^{t_2} \frac 1{\tau^k} d\tau\overset {t_1, t_2\rightarrow \infty }\longrightarrow 0.\end{multline*}
Hence we deduce by completness that $\exists \tilde f_\pm\in H^s(\R)$ such that
$$\|u(t,x)-  e^{it(\partial_x^2 + V)} \tilde f_\pm\|_{H^s(\R)}\overset{t\rightarrow \pm \infty}\longrightarrow 0.$$
We conclude since by Proposition \ref{upgrades} there exist
$f_\pm\in H^s(\R)$ such that
$$\|e^{it(\partial_x^2 + V)} \tilde f_\pm - e^{it\partial_x^2 } f_\pm\|_{H^s(\R)}
\overset{t\rightarrow \pm \infty}\longrightarrow 0.$$

Next we focus on the proof of \eqref{momentumprecise}.
The following proposition will be crucial for our purpose.
\begin{prop}\label{part}
Let $u(t,x)$ be the unique global solution to \eqref{NLS}
with $f\in \Sigma_s({\R})$, then
\begin{equation}\label{equivfac}\lim_{R\rightarrow \infty}
(\limsup_{t\rightarrow \infty} \int_{|x|>Rt}  \big( \frac{x}{t}\big)^{2s} |u(t,x)|^2   dx)=0.\end{equation}
\end{prop}
\begin{proof}
As a preliminary fact we notice that \begin{equation}\label{largeNLS}
\big (\limsup_{t\rightarrow \infty} \int_{|x|>Rt} |\partial_x^j u(t,x)|^2 dx
\big )\overset {R\rightarrow \infty} \longrightarrow 0, \quad \forall j=0, \dots ,s.
\end{equation}
The proof of this fact follows from 
\begin{equation}\label{catben}\|\partial_x^j(u(t,x)) - \frac {e^{i\frac{x^2}{4t}}}{\sqrt {2i t}} 
{\mathcal F}(\partial_x^j f_+)(\frac {x}{2t})\|_{L^2({\R})}\overset{t\rightarrow \infty}\rightarrow 0
\end{equation}
(here $\mathcal F$ denotes the Fourier transform \eqref{fourier}) which in turn is a consequence of
$$\|u(t,x)- e^{it\partial_{x}^2} f_+\|_{H^s({\R})}\overset {t\rightarrow \infty} \longrightarrow 0$$
and the asymptotic description of free waves \eqref{saymptoticbehavior}. 
Then from \eqref{catben} we get 
\begin{equation*}\lim_{t\rightarrow \infty} 
\int_{|x|>Rt} |\partial_x^j u(t,x)|^2 dx= \int_{|x|>\frac R2} 
|{\mathcal F} (\partial_x^j f_+)|^2 dx
\end{equation*}
and hence \eqref{largeNLS} follows from
\begin{equation*}
\lim_{R\rightarrow \infty} \int_{|x|>\frac R2} 
|{\mathcal F} (\partial_x^j f_+)|^2 dx=0.
\end{equation*}
Next we work by induction on $s$ and we show the following fact 
(that in turn implies the conclusion of Proposition~\ref{part} if one takes $h=0$):
\begin{equation}\label{inductch}
\lim_{R\rightarrow \infty} (\limsup_{t\rightarrow \infty}
\int_{|x|>Rt} \big(\frac x t\big)^{2s-2 h} |\partial_x^{h} u(t,x))|^2 dx)=0,
\quad \forall h=0,\dots, s.
\end{equation}
Notice that for $h=s$ the estimate \eqref{inductch} follows by 
\eqref{largeNLS}.\\
\\
 
Next we prove \eqref{inductch}, by induction on $s$ and by
backward induction 
on $h$, for $h=1,\dots, s$. Namely we assume
\begin{equation}\label{ass1234}
\lim_{R\rightarrow \infty} (\limsup_{t\rightarrow \infty}
\int_{|x|>Rt} \big(\frac x t\big)^{2s-2k} |\partial_x^{j} u(t,x))|^2 dx)=0, \quad \forall
0\leq j< k\leq  s
\end{equation}
and
\begin{equation}\label{ass123}
\lim_{R\rightarrow \infty} (\limsup_{t\rightarrow \infty}
\int_{|x|>Rt} \big(\frac x t\big)^{2s-2\bar h} |\partial_x^{\bar h} u(t,x))|^2 dx)=0
\end{equation}
for some
$\bar h\in \{2,\dots, s\}$. Then we prove \eqref{inductch}
for $h=\bar h -1$.
Notice that we have the following elementary identity for any smooth function
\begin{multline*}\int_{\R} x^{2s-2\bar h +2} |\partial_x^{\bar h-1} v(x)|^2
dx=-\int_{\R} x^{2s-2\bar h +2} \partial_x^{\bar h} v(x) \partial_x^{\bar h-2} \bar v(x) dx
\\- (2s-2\bar h +2) \int_{\R} x^{2s-2\bar h +1} \partial_x^{\bar h-1} v(x)  \partial_x^{\bar h-2} \bar v(x)  dx \end{multline*}
which implies
\begin{multline*}
\int_{\R} x^{2s-2\bar h +2} |\partial_x^{\bar h-1} v(x)|^2 
dx \lesssim \sqrt{\int_{\R} x^{2s-2\bar h} | \partial_x^{\bar h} v(x)|^2 dx}\sqrt {\int_{\R} x^{2s-2\bar h +4}|\partial_x^{\bar h-2} v(x)|^2 dx}\\
+ \sqrt{\int_{\R} x^{2s-2\bar h} |\partial_x^{\bar h-1} v(x)|^2 dx}\sqrt {\int_{\R} x^{2s-2\bar h +2}|\partial_x^{\bar h-2} v(x)|^2 dx}
\,.\end{multline*}
and hence for every $t>0$
\begin{multline}\label{intpp}
\int_{\R} \big( \frac xt\big)^{2s-2\bar h +2} |\partial_x^{\bar h-1} v(x)|^2 
dx \\\lesssim \sqrt{\int_{\R} \big(\frac xt\big)^{2s-2\bar h} | \partial_x^{\bar h} v(x)|^2 dx}\sqrt {\int_{\R} \big( \frac xt\big)^{2s-2\bar h +4}|\partial_x^{\bar h-2} v(x)|^2 dx}\\
+ \frac 1t \sqrt{\int_{\R} \big( \frac xt\big)^{2s-2\bar h} |\partial_x^{\bar h-1} v(x)|^2 dx}\sqrt {\int_{\R} \big(\frac xt\big)^{2s-2\bar h +2}|\partial_x^{\bar h-2} v(x)|^2 dx}
\,.\end{multline}
Next introduce a cut-off function $\varphi\in C^\infty(\R)$ such that
$$\varphi(x)=0 \quad \forall\, |x|<\frac 12 , \quad
\varphi(x)=1 \quad \forall\, |x|>1$$ and define
$$\varphi_R(t,x)=\varphi(\frac x{Rt}).$$
We claim that for every fixed $R>0$
\begin{equation}\label{defmormont}
\int_\R x^{2s-2k} |\partial_x^j (\varphi_R (t,x) u(t,x))|^2 dx
\lesssim \langle t \rangle^{2s-2k}, \quad \forall \,0\leq j\leq k\leq s
\end{equation}
and also
\begin{equation}\label{defmormonth}
\lim_{R\rightarrow \infty} (\limsup_{t\rightarrow \infty}
\int_\R \big(\frac x t\big)^{2s-2\bar h} |\partial_x^{\bar h} (\varphi_R (t,x) u(t,x))|^2 dx)=0.
\end{equation}
Once \eqref{defmormont} and \eqref{defmormonth} are established then we get
from \eqref{intpp}, where we choose
$v(x)= \varphi_R (t,x) u(t,x)$,
\begin{equation}\label{fofab}
\lim_{R\rightarrow \infty} (\limsup_{t\rightarrow \infty} 
\int_{\R} \big( \frac x t\big)^{2s-2\bar h +2} |\partial_x^{\bar h-1} (\varphi_R (t,x) u(t,x))|^2 
dx)=0.\end{equation}
Next notice that 
$$\lim_{t\rightarrow \infty}
\|\partial_x^{\bar h-1-j} \varphi_R (t,x)\|_{L^\infty(\R)} =0, \quad \forall j=0,\dots, \bar h-2 $$
and by \eqref{induct1}
$$\int_{\R} x^{2s-2\bar h +2} |
\partial^{j}_x u(t,x))|^2 dx\lesssim \langle t\rangle^{2s-2\bar h +2}$$
and hence
\begin{multline}\label{forpal}
\limsup_{t\rightarrow \infty} 
\int_{\R} \big( \frac x t\big)^{2s-2\bar h +2} |\partial_x^{\bar h-1-j} \varphi_R (t,x) \partial^{j}_x u(t,x))|^2 dx=0, \\ \forall j=0,\dots, \bar h-2 
\end{multline}
By using \eqref{forpal} and the Leibnitz rule to compute
$\partial_x^{\bar h-1} (\varphi_R (t,x) u(t,x))$ we get
$$
\limsup_{t\rightarrow \infty}
\int_{\R} \big(\frac xt\big)^{2s-2\bar h +2} \big( |\varphi_R (t,x)|^2 |\partial_x^{\bar h-1} u(t,x)
|^2 - |\partial_x^{\bar h-1}  (\varphi_R (t,x)u(t,x))
|^2)\big) dx=0.$$
Hence by \eqref{fofab} we get
$$
\lim_{R\rightarrow \infty} \big(\limsup_{t\rightarrow \infty}
\int_{\R} \big(\frac xt\big)^{2s-2\bar h +2} |\varphi_R (t,x)|^2 |\partial_x^{\bar h-1} u(t,x)
|^2 dx \big)=0$$
and due to the cut-off property of $\varphi$ it implies
$$
\lim_{R\rightarrow \infty} \big( \limsup_{t\rightarrow \infty}
\int_{|x|>Rt} \big(\frac xt\big)^{2s-2\bar h +2} |\partial_x^{\bar h-1} u(t,x)
|^2 dx \big) =0.$$
Concerning the proof of \eqref{defmormont}
it follows, once the Leibniz rule is used to compute the derivative of a product,
by \eqref{induct1}.
For the proof of 
\eqref{defmormonth} notice that again by the Leibnitz rule
and the fact that
$$\lim_{t\rightarrow \infty}
\|\partial_x^{\bar h-j} \varphi_R (t,x)\|_{L^\infty(\R)} =0, \quad \forall j=0,\dots, \bar h-1, $$
we get
$$
\limsup_{t\rightarrow \infty}
\int_{\R} \big(\frac xt\big)^{2s-2\bar h} \big( |\varphi_R (t,x)|^2 |\partial_x^{\bar h} u(t,x)
|^2 - |\partial_x^{\bar h}  (\varphi_R (t,x)u(t,x))
|^2)\big) dx=0.$$
We conclude the proof of \eqref{defmormonth}
since due to the cut-off property of $\varphi$ we get
\begin{multline*}\lim_{R\rightarrow \infty} \big(\limsup_{t\rightarrow \infty}
\int_{\R} \big(\frac xt\big)^{2s-2\bar h} \big( |\varphi_R (t,x)|^2 |\partial_x^{\bar h} u(t,x)
|^2dx\big)
\\\leq \lim_{R\rightarrow \infty} \big(\limsup_{t\rightarrow \infty}
\int_{2|x|> Rt} \big(\frac xt\big)^{2s-2\bar h} |\partial_x^{\bar h} u(t,x)
|^2 dx\big)=0\end{multline*}
where at the last step we have used the inductive assumption \eqref{ass123}.
\\

The last step is the proof of \eqref{inductch} for $h=0$ provided that it is known for $h=1$
(this fact in turn has been established above).

Then we shall prove the following equivalent version of \eqref{inductch}
for $h=0$:
\begin{align}\label{inductchlast}
&\forall \epsilon>0 \hbox{ } \exists \tilde R_\epsilon>0, \tilde t_\epsilon>0 \hbox { s.t. }\\\nonumber
\int_{|x|>\tilde R_\epsilon t} |x|^{2s}| u(t,x)|^2& \lesssim \epsilon t^{2s}, \quad 
\forall t>\tilde t_\epsilon
\end{align}
by using the fact (this is an equivalent version of \eqref{inductch} for $h=1$)
\begin{align}\label{inductchlastprime}
&\forall \epsilon>0 \hbox{ } \exists R_\epsilon>0, t_\epsilon>0 \hbox { s.t. }\\\nonumber
\int_{|x|>R_\epsilon t} |x|^{2s-2}|\partial_x u(t,x)|^2& \lesssim \epsilon t^{2s-2}, \quad 
\forall t>t_\epsilon.
\end{align}
We introduce
a function $\varphi\in C^\infty(\R)$ such that
$$\varphi(x)=0 \quad \forall |x|<1, \quad
\varphi(x)=1 \quad \forall  |x|>2$$
and the corresponding rescaled function $\varphi_M(x)=\varphi\left (\frac xM
\right )$.
By direct computation we get
$$\frac d{dt} \int_{\R} \varphi_M(x) x^{2s} |u(t,x)|^2 dx= 2 \Im 
\int_{\R} \partial_x (\varphi_M(x) x^{2s})\bar u  \partial_x u dx$$
and hence
\begin{multline*}\big |\frac d{dt} \int_{\R} \varphi_M(x) x^{2s} |u(t,x)|^2 dx\big |\lesssim \frac 1M\int_M^{2M} 
x^{2s} |\bar u(t,x)| |\partial_x u(t,x)| dx\\
+ \int_M^{\infty} |x|^{2s-1} |\bar u(t,x)| |\partial_x u(t,x)| dx
\lesssim \int_M^{\infty} |x|^{2s-1} |\bar u(t,x)| |\partial_x u(t,x)| dx.
\end{multline*}
As a consequence we get
\begin{multline*}
\int_{\R} \varphi_M(x) x^{2s} |u(t,x)|^2 dx\lesssim \int_{\R} \varphi_M(x) x^{2s} |u(\bar t,x)|^2 dx
\\
+\int_{\bar t}^t \sqrt{\big (\int_M^{\infty} x^{2s-2} |\partial_x u(\tau, x)|^2 dx\big)}
\sqrt{\big (\int_{\R} x^{2s} |u(\tau, x)|^2 dx\big)}
d\tau,\quad \forall\, t>\bar t.
\end{multline*}
Next we fix 
$\bar t=t_\epsilon$ and $M=R_\epsilon t$ (where $t_\epsilon, R_\epsilon$
are given by \eqref{inductchlastprime}), then by \eqref{inductchlastprime} and the upper bound \eqref{upperboundt2s} we get
$$\int_{|x|>R_\epsilon t} 
x^{2s} |u(t,x)|^2 dx\lesssim \int_{\R}
x^{2s} |u(t_\epsilon,x)|^2 dx + \epsilon \int_{t_\epsilon}^t \tau^{2s-1}d\tau$$
which implies 
$$\int_{|x|>R_\epsilon t} 
x^{2s} |u(t,x)|^2 dx\lesssim \int_{\R}
x^{2s} |u(t_\epsilon,x)|^2 dx + \epsilon t^{2s} - \epsilon t_\epsilon^{2s}.$$
We conclude provided that we multiply 
the previous inequality by $t^{-2s}$ and we take  $t\gg t_\epsilon$.
\end{proof}
We can now conclude the proof of \eqref{momentumprecise}.
Fix $\epsilon>0$. Then thanks to  \eqref{equivfac}, we have that there exists $t_\epsilon, R_\epsilon>0$ such that
\begin{equation}\label{externalcone}
\int_{|x|>R_\epsilon t}  \big( \frac{x}{t}\big)^{2s} |u(t,x)|^2  dx< \frac{\epsilon}{4}, \quad 
\forall\,  t> t_\epsilon
\end{equation}
and also
\begin{equation}\label{ddd}
2^{2s}\big|\int_{|x|<R_{\epsilon}} |x|^{2s}|\hat f_+(x)|^2 dx-\int_\R |\partial_x^s f_+(x)|^2 dx\big|<\frac{\epsilon}{2}\,.
\end{equation}
Notice that by combining \eqref{alglin} 
and \eqref{saymptoticbehavior}  we get
\begin{multline*}\left \|\big(\frac{x}{t}\big)^s\big (u(t,x) - \frac {e^{i\frac{x^2}{4t}}}{\sqrt {2i t}} \hat f_+(\frac {x}{2t})
\big)
\right \|_{L^2(|x|<R_\epsilon t)}\\\leq 
\left \|\big(\frac{x}{t}\big)^s\big (u(t,x) - e^{it\partial_{x}^2} f_+\big)
\right \|_{L^2(|x|<R_\epsilon t)}\\+
\left \|\big(\frac{x}{t}\big)^s\big (e^{it\partial_{x}^2} f_+ -
\frac {e^{i\frac{x^2}{4t}}}{\sqrt {2i t}} \hat f_+(\frac {x}{2t})
\big)
\right \|_{L^2(|x|<R_\epsilon t)}
\\\leq R_\epsilon^s \left \|u(t,x) - e^{it\partial_{x}^2} f_+
\right \|_{L^2(\R)} + R_\epsilon^s \left \|e^{it\partial_{x}^2} f_+ -
\frac {e^{i\frac{x^2}{4t}}}{\sqrt {2 i t}} \hat f_+(\frac {x}{2t})
\right \|_{L^2({\R})}
\overset{t\rightarrow \infty}\longrightarrow 0.\end{multline*}
Summarizing we get
\begin{equation*}
\left \|\big(\frac{x}{t}\big)^s\big (u(t,x) - \frac {e^{i\frac{x^2}{4t}}}{\sqrt {2 i t}} \hat f_+(\frac {x}{2t})
\big)
\right \|_{L^2(|x|<R_\epsilon t)}
\overset{t\rightarrow \infty}\longrightarrow 0.
\end{equation*}
Therefore, thanks to \eqref{externalcone}, we obtain that there exists $\tilde{t}_{\epsilon}\geq t_\epsilon$ such that for $t>\tilde{t}_{\epsilon}$,
$$
\Big|\int_{\R}  \big( \frac{x}{t}\big)^{2s} |u(t,x)|^2  dx-
\int_{|x|<R_\epsilon t}  \big (\frac{x}{t}\big )^{2s} \frac {1}{2t} |\hat f_+(\frac {x}{2t})|^2  \Big|<\frac{\epsilon}{2}\,.
$$
By a change of variable we have
\begin{equation}\label{externaleasy}
\int_{|x|<R_\epsilon t}  \big (\frac{x}{t}\big )^{2s} \frac {1}{2t} |\hat f_+(\frac {x}{2t})|^2  dx =2^{2s} \int_{|x|<R_\epsilon} |x|^{2s}|\hat f_+(x)|^2 dx.
\end{equation}
Therefore coming back to \eqref{ddd}, we get 
$$
\Big|\int_{|x|<R_\epsilon t}  \big (\frac{x}{t}\big )^{2s} \frac {1}{2t} |\hat f_+(\frac {x}{2t})|^2 dx-2^{2s}\int_\R |\partial_x^s f_+(x)|^2 dx\Big|<\frac{\epsilon}{2}
$$
and therefore, we finally obtain that for  $t>\tilde{t}_{\epsilon}$,
$$
\Big|\int_{\R}  \big( \frac{x}{t}\big)^{2s} |u(t,x)|^2  dx-2^{2s}\int_\R |\partial_x^s f_+(x)|^2 dx\Big|<\epsilon\,.
$$


\begin{thebibliography}{10}


\bibitem{Ag} S. Agmon,
 {\em Spectral properties of {S}chr\"{o}dinger operators and scattering
              theory},
Ann. Scuola Norm. Sup. Pisa Cl. Sci. (4),
2, (1975), n. 2, 151--218

\bibitem{BV} V. Banica, N. Visciglia, {\em Scattering for {NLS} with a delta potential}, J. Differential Equations, 260, (2016), n. 5, 4410--4439.

\bibitem{carles} R. Carles, {\em Nonlinear {S}chr\"{o}dinger equation with time dependent potential}, Commun. Math. Sci., 9, (2011), n. 4, 937--964.

\bibitem{Ca} T. Cazenave {\em Semilinear {S}chr\"{o}dinger equations}, Courant Lecture Notes in Mathematics, 10, New York University, Courant Institute of Mathematical Sciences, New York; American Mathematical Society, Providence, RI, 2003

\bibitem{CGV} S. Cuccagna, V. Georgiev, N. Visciglia, {\em Decay and scattering of small solutions of pure power {NLS} in {$\Bbb R$} with {$p > 3$} and with a potential}, Comm. Pure Appl. Math., 67, (2014) n. 6, 957--981.

\bibitem{DGS}  M. Duerinckx, A. Gloria, C. Shirley {\em Approximate normal forms via Floquet-Bloch theory. Part 1: Nehoro\v{s}ev stability for linear waves in quasiperiodic media},
arXiv:1809.07106 

\bibitem{Do} B. Dodson, {\em Global well-posedness and scattering for the defocusing,  {$L^2$} critical, nonlinear {S}chr\"{o}dinger equation when  {$d=1$}}, Amer. J. Math., 138, (2016) n. 2,
531--569.
     
\bibitem{D} J. D. Dollard, {\em Quantum-mechanical scattering theory for short-range and {C}oulomb interactions}, Rocky Mountain J. Math., 1, (1971), n. 1, 5--88.

\bibitem{ES} A. Elgart, B. Schlein,  {\em Adiabatic charge transport and the {K}ubo formula for {L}andau-type {H}amiltonians},  Comm. Pure Appl. Math., 57, (2004), n. 5, 590--615.
     
\bibitem{En} V. Enss, {\em Asymptotic completeness for quantum mechanical potential
scattering. {I}. {S}hort range potentials}, Comm. Math. Phys.,
61, (1978), n. 3, 285--291

\bibitem{FV} L. Forcella, N. Visciglia, {\em Double scattering channels for 1{D} {NLS} in the energy space and its generalization to higher dimensions}, J. Differential Equations, 264, (2018), n. 2, 929--958.

\bibitem{GV} J. Ginibre, G. Velo,  {\em On a class of nonlinear {S}chr\"{o}dinger equations. {II}. {S}cattering theory, general case}, J. Functional Analysis, 32, (1979), n. 1, 33--71.

\bibitem{L} D. Lafontaine,  {\em Scattering for {NLS} with a potential on the line}, Asymptot. Anal., 100, (2016), n. 1-2, 21--39.

\bibitem{N} K. Nakanishi, {\em Energy scattering for nonlinear {K}lein-{G}ordon and {S}chr\"{o}dinger equations in spatial dimensions {$1$} and {$2$}}, J. Funct. Anal., 169, (1999), n. 1, 201--225.
      
\bibitem{PV} F. Planchon, L. Vega, {\em Bilinear virial identities and applications}, Ann. Sci. \'{E}c. Norm. Sup\'{e}r. (4), 42, (2009), n.2, 261--290.
     
\bibitem{RS} C. Radin, B. Simon,  {\em Invariant domains for the time-dependent {S}chr\"{o}dinger equation}, J. Differential Equations, 29, (1978), n. 2, 289--296.
     
\bibitem{RS2} M. Reed, B. Simon, Methods of modern mathematical physics. {II}. {F}ourier analysis, self-adjointness, Academic Press [Harcourt Brace Jovanovich, Publishers], 
New York-London, 1975.

\bibitem{RS4} M. Reed, B. Simon,   Methods of modern mathematical physics. {IV}. {A}nalysis of operators, Academic Press [Harcourt Brace Jovanovich, Publishers], New York-London,1978.

\bibitem{S} B. Simon, {\em Absence of ballistic motion}, Comm. Math. Phys., 134, (1990), n. 1, 209--212.

\bibitem{TV} N. Tzvetkov, N. Visciglia, {\em Small data scattering for the nonlinear {S}chr\"{o}dinger equation on product spaces}, Comm. Partial Differential Equations, 37, (2012), n.1, 125--135.
    
\bibitem{VV2} L. Vega, N. Visciglia,  {\em On the local smoothing for the {S}chr\"{o}dinger equation}, Proc. Amer. Math. Soc., 135, (2007), n.1, 119--128.
  
\bibitem{VV} L. Vega, N.Visciglia,  {\em On the equipartition of energy for the critical {${\rm NLW}$}}, J. Funct. Anal., 255, (2008), n.3, 726--754.

\bibitem{VV1} L.Vega, N. Visciglia,  {\em Asymptotic lower bounds for a class of {S}chr\"{o}dinger equations}, Comm. Math. Phys., 279, (2008), n. 2, 429--453.

\bibitem{V} N. Visciglia,  {\em On the decay of solutions to a class of defocusing {NLS}}, Math. Res. Lett., 16, (2009), n.5, 919--926.

\bibitem{Z} Z. Zhao, {\em Ballistic transport in one-dimensional quasi-periodic continuous {S}chr\"{o}dinger equation}, J. Differential Equations, 262, 2017, n. 9, 4523--4566.

\end{thebibliography}
\end{document}